\documentclass[10pt, a4paper]{amsart}

\usepackage[utf8]{inputenc}
\usepackage[english]{babel}

\usepackage{amsthm}
\usepackage{amstext}
\usepackage{amsmath}
\usepackage{amsfonts}
\usepackage{amssymb}
\usepackage{mathtools}

\usepackage{tikz}
\usepackage{graphics}
\usepackage{wrapfig}
\usetikzlibrary{matrix,arrows,decorations.pathmorphing}

\usepackage[a4paper]{geometry}
\usepackage{tikz-cd}
\usepackage{epigraph}

\usepackage{hyperref}
\usepackage{cite}
\usepackage{cleveref}
\crefformat{subsection}{\S#2#1#3}

\usepackage{enumerate}
\usepackage{enumitem}

\author{Matteo Tamiozzo}
\title{Congruences of modular forms and modularity of Tate--Shafarevich classes}

\swapnumbers
\newtheorem{thm}[subsubsection]{Theorem}     %definizione ambiente teorema
\theoremstyle{plain}                    %stile corsivo
\newtheorem{conj}[subsubsection]{Conjecture} %definizione ambiente congettura
\newtheorem{prop}[subsubsection]{Proposition}    %definizione ambiente proposizione
\newtheorem{corol}[subsubsection]{Corollary}     %definizione ambiente corollario
\newtheorem{lem}[subsubsection]{Lemma}         %definizione ambiente lemma
\theoremstyle{definition}               %stile roman

\theoremstyle{remark}                   %stile per osservazioni
\newtheorem{rem}[subsubsection]{Remark}      %definizione ambiente osservazione

\newcommand{\C}{\mathbf{C}}     %complex
\newcommand{\Q}{\mathbf{Q}}      %rational
\newcommand{\Z}{\mathbf{Z}}    %integer

\DeclareFontFamily{U}{wncy}{}
    \DeclareFontShape{U}{wncy}{m}{n}{<->wncyr10}{}
    \DeclareSymbolFont{mcy}{U}{wncy}{m}{n}
    \DeclareMathSymbol{\Sh}{\mathord}{mcy}{"58}

\numberwithin{equation}{subsection}

\begin{document}

\begin{abstract}
We prove, under suitable assumptions, that $p$-torsion Tate--Shafarevich classes for elliptic curves over the rationals are visible in quotients of Jacobians of modular curves, as predicted by a conjecture of Jetchev--Stein. The key ingredient is the non-triviality of the Bertolini--Darmon bipartite Kolyvagin system, which implies that suitable cohomology classes of the system form a basis of the Selmer group modulo $p$.
\end{abstract}

\maketitle

\tableofcontents

\section{Introduction}

Let $E$ be an elliptic curve over a number field $K$, with $L$-function $L(E/K, s)$; we denote by $n$ the rank of $E(K)$, and by $|\Sh(E/K)|$ the cardinality of the Tate--Shafarevich group $\Sh(E/K)$. The Birch and Swinnerton-Dyer conjecture predicts that $\mathrm{ord}_{s=1}L(E/K, s)=n$, and
\begin{equation}\label{svf}
\frac{L^{(n)}(E/K, 1)}{n!P(E/K)R(E/K)}=|\Sh(E/K)|,
\end{equation}
where the regulator $R(E/K)$ and the period $P(E/K)$ are defined in \cite[Definitions 1.6, 2.9]{gro11}.

\subsection{Bertolini--Darmon's bipartite Kolyvagin system} Assume in the rest of this introduction that the elliptic curve $E$ is defined over $\Q$, so that $L(E/\Q, s)=L(f, s)$ for a newform $f \in S_2(\Gamma_0(N))$, and let $K/\Q$ be an imaginary quadratic field. In the seminal work \cite{bd05}, Bertolini--Darmon introduced a strategy to study the Birch and Swinnerton-Dyer conjecture for $E/K$ based on congruences of modular forms.
\subsubsection{} Let us first recall the basic idea to show that $E(K)$ is finite if $L(E/K, 1)\neq 0$; more details on the argument and the required assumptions can be found in the body of the text, and in \cite{bd05}, \cite{nek12}. For a fixed prime $p>3$, we consider newforms $f_\ell\in S_2(\Gamma_0(N\ell))$ congruent to $f$ modulo (a prime above) $p$, for suitable auxiliary primes $\ell$ inert in $K$, called admissible primes (cf. \cref{bip-set}). Heegner points on the abelian varieties $A_\ell$ attached to the modular forms $f_\ell$ yield cohomology classes $z_\ell^{BD} \in H^1(K, E[p])$ which are possibly ramified at $\ell$, as $A_\ell$ has bad reduction at $\ell$. More precisely, assuming that the left-hand side of \eqref{svf} for $n=0$ is a $p$-adic unit, one shows via the so-called first reciprocity law that $z_\ell^{BD}$ is not a Selmer class, as it fails to satisfy the relevant local condition at $\ell$. Varying $\ell$, this allows to construct enough classes in $H^1(K, E[p])\smallsetminus \mathrm{Sel}(K, E[p])$ to force, by global duality, the vanishing of $\mathrm{Sel}(K, E[p])$. In particular, $E(K)$ is finite.

\subsubsection{} More generally, Bertolini--Darmon construct classes $z_\mathfrak{l}^{BD}\in H^1(K, E[p])$, indexed by suitable products of admissible primes, which are part of a bipartite Kolyvagin system $Z^{BD}$ for $E[p]$. The reader can consult \cref{bipkol}, \cref{constr-bd} for the definitions, and Remark \ref{termin} if surprised by our terminology. Using such classes - and working modulo higher powers of $p$ - one can prove the following implication, under suitable assumptions:
\begin{equation*}
\mathrm{ord}_{s=1}L(E/K, s)\leq 1 \Rightarrow n=\mathrm{ord}_{s=1}L(E/K, s), \text{ and } v_p\left(\frac{L^{(n)}(E/K, 1)}{n!P(E/K)R(E/K)}\right)\geq v_p(|\Sh(E/K)[p^\infty]|).
\end{equation*}
We refer the reader to \cite{bbv16}, \cite{blv}, as well as to \cite[\S 1]{tam21} for a similar result for Hilbert modular forms of parallel weight two and references to further related works. In addition, as highlighted in \cite{tam21}, in order to conclude that the $p$-adic valuations of the two sides of \eqref{svf} are equal one needs to know that
\begin{equation}\label{rib2}
L^{alg}(E/K, 1) \equiv 0 \pmod p \Rightarrow \mathrm{Sel}(K, E[p]) \neq 0.
\end{equation}
More precisely, one needs an analogue of this implication for modular forms obtained raising the level of $f$ modulo $p$, cf. \cite[Theorem 1.5, Remark 6.11]{tam21} and \cref{zbdntriv}.
\subsection{A basis of $\mathrm{Sel}(K, E[p])$} This note originated from our attempt to study the implication \eqref{rib2}; our starting point is the simple observation that the first reciprocity law implies that 
each $z_{\ell}^{BD}$ is a Selmer class if $L^{alg}(E/K, 1)$ is not a $p$-adic unit. Under certain assumptions implying that $Z^{BD}$ is non-trivial (i.e. contains at least one non-zero element) we prove in Theorem \ref{bd-gensel} that suitable classes $z_{\mathfrak{l}}^{BD}$ form a basis of $\mathrm{Sel}(K, E[p])$. A similar problem for Kolyvagin's system of Heegner points was raised by Gross \cite[\S 11]{gro91} and studied by Zhang W. \cite[\S 8.2]{zha14}.

Assume that $L^{alg}(E/K, 1) \equiv 0 \pmod p$. If $Z^{BD}$ is non-trivial, then \eqref{rib2} holds (cf. Proposition \ref{nontriv-conv} and \cref{zbdntriv}) and, by Proposition \ref{gen-sel-abs}, one of the classes $z_\mathfrak{l}^{BD}$ belongs to $\mathrm{Sel}(K, E[p]) \smallsetminus \{0\}$. However, let us emphasize that, at present, the only way we know to prove the non-triviality of $Z^{BD}$ consists in deducing it from a generalization of \eqref{rib2}, cf. \cref{zbdntriv}.

\subsection{Modularity of Tate--Shafarevich classes} We apply Theorem \ref{bd-gensel} to the study of visibility of Tate--Shafarevich classes; some background on the subject can be found in \cref{back-vis}. Recall that $c \in H^1(K, E)$ is visible in an abelian variety $A$ in which $E$ is embedded if the image of $c$ in $H^1(K, A)$ is trivial. Jetchev--Stein \cite[Conjecture 7.1.1]{js07} conjectured that every class $c \in \Sh(E/\Q)$ is visible in a quotient of the Jacobian of a modular curve $X_0(M)$; we call a class $c$ with this property \emph{modular}. Concretely, this tells us that $c$ can be realized as an $E$-torsor contained in a quotient of $\mathrm{Jac}(X_0(M))$.

As mentioned above, each class $z_\ell^{BD}$ comes from a $K$-point on a modular abelian variety $A_\ell$ with $E[p]\subset A_\ell$; it follows that the image of $z_\ell^{BD}$ in $H^1(K, E)$ is visible in the quotient $(E \times A_\ell)/E[p]$ (where $E[p]$ is embedded anti-diagonally; cf. \cref{mainproof}). Generalizing this observation, we prove new cases of the conjecture of Jetchev--Stein (previously known only for classes of order at most 3). A special case of our main theorem is the following; see Theorem \ref{mainthm} and Corollary \ref{cor-jsconj} for more general results.

\begin{thm}
Let $E/\Q$ be an elliptic curve of squarefree conductor $N$, without complex multiplication. Let $p>3$ be a prime of good ordinary reduction. Assume that $\bar{\rho}: \mathrm{Gal}(\bar{\Q}/\Q)\rightarrow \mathrm{Aut}_{\mathbf{F}_p}(E[p])$ is surjective and ramified at every prime $q \mid N$ such that $q \equiv \pm 1 \pmod p$. Then every class in $\Sh(E/\Q)[p]$ is modular.
\end{thm}

\subsubsection{} Note that there is no restriction on the rank of $E(\Q)$ in the theorem. The hypotheses are imposed to ensure on the one hand that the bipartite Kolyvagin system $Z^{BD}$ can be constructed, and on the other hand that it is non-trivial. For the reader's convenience, we recall that the most important ingredients to achieve these goals are the following (cf. also the introduction of \cite{tam21}).
\begin{enumerate}
\item Results on the cohomology with $\Z/p\Z$-coefficients of Shimura curves and Shimura sets (multiplicity one, Ihara's lemma) used to construct $Z^{BD}$.
\item The Iwasawa main conjecture - especially Skinner--Urban's divisibility - as well as a comparison of periods, employed to prove \eqref{modp-conv} (generalizing \eqref{rib2}) and to deduce the non-triviality of $Z^{BD}$.
\end{enumerate}
We have not attempted to make our assumptions optimal; still, they are easy to verify in practice, and they do hold true in a non-trivial example, illustrated in \cref{concrete}.

\subsection{Structure of the text} In \cref{sec-gensel} we study bipartite Kolyvagin systems for $E[p]$ and prove Theorem \ref{bd-gensel}. We formulate our arguments abstractly, hoping that this clarifies their structure and highlights the key point, i.e. the non-triviality of $Z^{BD}$. Along the way, we show that bipartite Kolyvagin systems for $E[p]$ are essentially unique, and their existence is equivalent to the $p$-parity conjecture for $E/K$; cf. \cref{bipvspar}. We also take the opportunity to briefly illustrate how the Birch and Swinnerton-Dyer conjecture for $E/K$ - including the $p$-part of the special value formula \eqref{svf} - in analytic rank at most one, as well as $p$-converse theorems, can be deduced from the non-triviality of $Z^{BD}$ (without using other Kolyvagin systems). Many of the relevant arguments, collected in \cref{weak-bsd} and \cref{bd-prim}, can be found, more or less disguised, scattered throughout the literature \cite{bbv16}, \cite{blv}, \cite{how06}, \cite{ki24}, \cite{swe22}, \cite{zha14}. The results mentioned above constitute a substantial part of the evidence currently supporting the Birch and Swinnerton-Dyer conjecture for $E/K$. The reader might find it of some interest that they all follow from one unifying principle (even though, of course, they have also been proved individually by other means, which we do not attempt to discuss here).

In \cref{back-vis}, which is purely expository, we recall some some basic notions and known results in the theory of visibility of Tate--Shafarevich classes, trying to make this text accessible to readers without previous exposure to the topic. Finally, in \cref{mod-sha-sect} we prove our main results, namely Theorem \ref{mainthm} and Corollary \ref{cor-jsconj}.

Sections \ref{sec-gensel} and \ref{back-vis} are independent of each other. The reader only interested in visibility of Tate--Shafarevich classes can start from \cref{back-vis}, then move on to \cref{mod-sha-sect} and refer back to \cref{sec-gensel} when needed. In order to understand the proofs in the last section, part of the general discussion in \cref{sec-gensel} is unnecessary: one only needs Theorem \ref{bd-gensel}, which is a linear algebra exercise using the two reciprocity laws (cf. \cref{bipkol}) and \cref{zbdntriv}.

\subsection*{Acknowledgements} This note originates from my attempt to explore the way paved by Bertolini--Darmon's work (especially \cite{bd05}), whose influence could hardly be overestimated. I would also like to gratefully acknowledge the influence of Christophe Cornut on this work: the germs of some of the ideas explained below were first developed in 2016 during my master thesis on \cite{zha14}, under Cornut's supervision. I am grateful to Christophe Cornut and Henri Darmon for independently suggesting that the results in \cref{sec-gensel} could have applications related to visibility of Tate--Shafarevich classes. I also wish to thank Luca Dall'Ava, Fred Diamond, Michele Fornea, Aleksander Horawa, Mohamed Moakher, Naomi Sweeting and Olivier Wittenberg for helpful discussions related to the topics studied in this article.

The work presented in this text was supported by the National Science Foundation under Grant
No. DMS-1928930, while I was in residence at the Mathematical Sciences Research Institute in Berkeley, California, during the Spring 2023 semester on Algebraic Cycles, $L$-values and Euler Systems. I wish to sincerely thank the staff for their support throughout my stay.

\section{Generating the mod $p$ Selmer group}\label{sec-gensel}

\subsection{Bipartite Kolyvagin systems modulo $p$}

\subsubsection{Setting}\label{bip-set}

Let $E/\Q$ be an elliptic curve of conductor $N$, corresponding to a newform
\begin{equation*}
f=\sum_{n=1}^\infty a_n q^n \in S_2(\Gamma_0(N)).
\end{equation*}
Let $p>3$ be a prime number not dividing $N$. We assume that $E$ has no complex multiplication, and that the Galois representation $\bar{\rho}: \mathrm{Gal}(\bar{\Q}/\Q)\rightarrow \mathrm{Aut}_{\mathbf{F}_p}(E[p])$ is surjective. Fix an imaginary quadratic field $K$ such that the prime factors of $pN$ are unramified in $K$, and write $N=N^-N^+$, where $q \mid N^-$ (resp. $q \mid N^+$) if $q$ is inert (resp. split) in $K$. Assume that $N^-$ is squarefree; the sign of the functional equation of $L(E/K, s)$ is
\begin{equation*}
\varepsilon(E/K)=(-1)^{|\{q \mid N^-\}|+1}.
\end{equation*}

A prime number $\ell\nmid pN$ inert in $K$ is called admissible if $\ell \not \equiv \pm 1 \pmod p$ and $a_\ell\equiv \pm (\ell+1) \pmod p$. We denote by $\mathcal{A}=\mathcal{A}^+\coprod \mathcal{A}^-$ the set of products $\mathfrak{l}=\ell_1\ell_2\cdots\ell_j$ of distinct admissible primes, where $\mathcal{A}^+=\{\mathfrak{l}=\ell_1\ell_2\cdots\ell_j : j \geq 0, |\{q \mid N^-\}|\not \equiv j \pmod 2\}$ and $\mathcal{A}^-=\{\mathfrak{l}=\ell_1\ell_2\cdots\ell_j : j \geq 0, |\{q \mid N^-\}|\equiv j \pmod 2\}$.

We will abusively denote the unique place of $K$ above an admissible prime $\ell \in \mathcal{A}$ still by $\ell$. We will need the following properties of admissible primes.
\begin{enumerate}
\item The $K_\ell$-module $E[p]$ splits (uniquely) as a direct sum $\Z/p\Z \oplus \Z/p\Z(1)$, and $H^1_{ur}(K_\ell, E[p])\simeq H^1(K_\ell, \Z/p\Z)\simeq \Z/p\Z$. We set $H^1_{tr}(K_\ell, E[p])=H^1(K_\ell, \Z/p\Z(1))\simeq \Z/p\Z$ (cf. \cite[Lemma 2.6]{bd05}, \cite[Lemma 2.2.1]{how06}).
\item For every $0 \neq c \in H^1(K, E[p])$ there are infinitely many $\ell \in \mathcal{A}$ such that $\mathrm{loc}_\ell(c)\neq 0$ (cf. \cite[Theorem 3.2]{bd05}).
\end{enumerate}

\subsubsection{Selmer groups}\label{sel-gps} The local condition at $\ell \in \mathcal{A}$ for the Selmer group $\mathrm{Sel}(K, E[p])$ is the unramified one. Given $\mathfrak{l}\in \mathcal{A}$, we define $\mathrm{Sel}_{(\mathfrak{l})}(K, E[p])$ changing the local condition to $H^1_{tr}(K_\ell, E[p])$ at every $\ell \mid \mathfrak{l}$. Furthermore, for $\ell \in \mathcal{A}$ not dividing $\mathfrak{l}$, using the local condition 0 (resp. $H^1(K_\ell, E[p])$) at $\ell$ we obtain a Selmer group $Sel_{(\mathfrak{l})\ell}(K, E[p])$ (resp. $Sel_{(\mathfrak{l})}^\ell(K, E[p])$). Thus, we have four Selmer groups fitting in the following rhombus:
\begin{center}
\begin{tikzcd}
& Sel_{(\mathfrak{l})}^\ell(K, E[p]) &\\
Sel_{(\mathfrak{l})}(K, E[p]) \arrow[ur] & & Sel_{(\mathfrak{l}\ell)}(K, E[p]) \arrow[ul] \\
& \arrow[ul] Sel_{(\mathfrak{l})\ell}(K, E[p]). \arrow[ur] &
\end{tikzcd}
\end{center}
By global duality we have $\mathrm{dim}(\mathrm{Sel}_{(\mathfrak{l})}^{\ell}(K, E[p]))-\mathrm{dim}(\mathrm{Sel}_{(\mathfrak{l})\ell}(K, E[p]))=1$, and either $\mathrm{Sel}_{(\mathfrak{l}\ell)}(K, E[p])=\mathrm{Sel}_{(\mathfrak{l})\ell}(K, E[p])$ and $\mathrm{Sel}_{(\mathfrak{l})}(K, E[p])=\mathrm{Sel}_{(\mathfrak{l})}^{\ell}(K, E[p])$, or $\mathrm{Sel}_{(\mathfrak{l}\ell)}(K, E[p])=\mathrm{Sel}_{(\mathfrak{l})}^{\ell}(K, E[p])$ and $\mathrm{Sel}_{(\mathfrak{l})}(K, E[p])=\mathrm{Sel}_{(\mathfrak{l})\ell}(K, E[p])$ (cf. \cite[Lemma 9]{gp12}, \cite[Proposition 2.2.9]{how06}). It follows that
\begin{equation}\label{chg-par}
\dim(\mathrm{Sel}_{(\mathfrak{l}\ell)}(K, E[p]))-\dim(\mathrm{Sel}_{(\mathfrak{l})}(K, E[p]))=\pm 1,
\end{equation}
and we have the following equivalence:
\begin{equation}\label{sel-1}
\mathrm{loc}_\ell: \mathrm{Sel}_{(\mathfrak{l})}(K, E[p])\rightarrow H^1_{ur}(K_\ell, E[p]) \text{ is surjective } \Leftrightarrow \dim(\mathrm{Sel}_{(\mathfrak{l}\ell)}(K, E[p]))=\dim(\mathrm{Sel}_{(\mathfrak{l})}(K, E[p]))-1.
\end{equation}
The implication from left to right will be crucial for us. Furthermore, let $r=\mathrm{dim}(\mathrm{Sel}(K, E[p]))$. Given $\mathfrak{l}=\ell_1\cdots \ell_j$ in $\mathcal{A}$, the dimension of $\mathrm{Sel}_{(\mathfrak{l})}(K, E[p])$ is at least $r-j$, so
\begin{equation}\label{nontr-sel}
j<r \Rightarrow \mathrm{Sel}_{(\mathfrak{l})}(K, E[p])\neq 0.
\end{equation}

\subsubsection{Bipartite Kolyvagin systems for $E[p]$}\label{bipkol} A bipartite Kolyvagin system $Z=Z^+\coprod Z^-$ for $E[p]$ consists of two collections
\begin{equation*}
Z^+=\{z_{\mathfrak{l}} \in \Z/p\Z, \mathfrak{l} \in \mathcal{A^+}\}, \; \; Z^-=\{z_{\mathfrak{l}} \in \mathrm{Sel}_{(\mathfrak{l})}(K, E[p]), \mathfrak{l} \in \mathcal{A^-}\}
\end{equation*}
satisfying the following relations (usually called first and second reciprocity law):
\begin{description}
\item[(RL1)] for every $\mathfrak{l}\in \mathcal{A}^-$ and every $\ell \mid \mathfrak{l}$, we have
\begin{equation*}
\mathrm{loc}_\ell(z_{\mathfrak{l}})=0 \in H^1_{tr}(K_\ell, E[p]) \Leftrightarrow z_{\mathfrak{l}/\ell}=0 \in \Z/p\Z;
\end{equation*}
\item[(RL2)] for every $\mathfrak{l}\in \mathcal{A}^-$ and every $\ell \in \mathcal{A}$, $\ell \nmid \mathfrak{l}$, we have
\begin{equation*}
\mathrm{loc}_\ell(z_{\mathfrak{l}})=0 \in H^1_{ur}(K_\ell, E[p]) \Leftrightarrow z_{\mathfrak{l}\ell}=0 \in \Z/p\Z.
\end{equation*}
\end{description}

The collection $Z_0$ of elements $z_\mathfrak{l}=0$ for every $\mathfrak{l}\in \mathcal{A}$ satisfies the above requirements; a bipartite Kolyvagin system for $E[p]$ different from $Z_0$ will be called non-trivial. We will recall later how an explicit bipartite Kolyvagin system $Z \neq Z_0$ can be constructed in our situation (under suitable assumptions), relying especially on the work of Bertolini--Darmon \cite{bd05} and Skinner--Urban \cite{su14}. In this section we will explore some consequences of the existence of a non-trivial $Z$. First of all, let us remark that either of the following assertions is equivalent to non-triviality of $Z$:
\begin{enumerate}[label=(\roman*)]
\item there is $\mathfrak{l}\in \mathcal{A}^+$ such that $z_\mathfrak{l}\neq 0$;
\item there is $\mathfrak{l}\in \mathcal{A}^-$ such that $z_\mathfrak{l}\neq 0$.
\end{enumerate}
Indeed, if (i) holds then by (RL1) for every $\ell \in \mathcal{A}$ not dividing $\mathfrak{l}$ we have $z_{\mathfrak{l}\ell}\neq 0$. Conversely, assume that $z_\mathfrak{l}\neq 0$ for some $\mathfrak{l}\in \mathcal{A}^-$. Choose an admissible prime $\ell \nmid \mathfrak{l}$ such that $\mathrm{loc}_\ell(z_{\mathfrak{l}})\neq 0$ (cf. \cref{bip-set}). By (RL2) we have $z_{\mathfrak{l}\ell}\neq 0$.

\begin{lem}\label{lem-selsmall}
Let $Z$ be a bipartite Kolyvagin system for $E[p]$.
\begin{enumerate}
\item If $\mathfrak{l}\in \mathcal{A}^+$ and $z_\mathfrak{l}\neq 0$ then $\mathrm{Sel}_{(\mathfrak{l})}(K, E[p])=0$.
\item If $\mathfrak{l}\in \mathcal{A}^-$ and $z_\mathfrak{l}\neq 0$ then $\mathrm{Sel}_{(\mathfrak{l})}(K, E[p])\simeq \Z/p\Z$.
\end{enumerate}
\end{lem}
\begin{proof}
The first point follows from global duality, (RL1) and the existence of admissible primes as in (2) of \cref{bip-set} (cf. \cite[Proposition 2.3.5]{how06} or \cite[Corollary 7.14]{tam21}). The second point follows from the first, (RL2) and \eqref{sel-1}, since by \cref{bip-set}(2) there exists $\ell \in \mathcal{A}$ not dividing $\mathfrak{l}$ such that $\mathrm{loc}_\ell(z_\mathfrak{l})\neq 0$.
\end{proof}

\begin{prop}\label{par-prop}
If there exists a non-trivial bipartite Kolyvagin system for $E[p]$ then, for every $\mathfrak{l}=\ell_1\cdots \ell_j \in \mathcal{A}$, we have
\begin{equation}\label{par-sel}
\mathrm{dim}(Sel_{(\mathfrak{l})}(K, E[p]))\equiv |\{q \mid N^-\}|+j+1 \pmod 2.
\end{equation}
\end{prop}
\begin{proof}
By the previous lemma and the remark preceding it there is $\mathfrak{l}_0 \in \mathcal{A}^+$ such that $\mathrm{Sel}_{(\mathfrak{l}_0)}(K, E[p])=0$. Furthermore, adding a prime factor to $\mathfrak{l}\in \mathcal{A}$ changes the parity of both sides of the congruence \eqref{par-sel} (see \eqref{chg-par}). Hence, if $\mathfrak{l}, \mathfrak{l}' \in \mathcal{A}$ and $\mathfrak{l}\mid \mathfrak{l}'$ then \eqref{par-sel} holds true for $\mathrm{Sel}_{(\mathfrak{l})}(K, E[p])$ if and only if it does for $\mathrm{Sel}_{(\mathfrak{l}')}(K, E[p])$. Now fix $\mathfrak{l}_0$ as above and take $\mathfrak{l} \in \mathcal{A}$. As \eqref{par-sel} holds for $\mathrm{Sel}_{(\mathfrak{l}_0)}(K, E[p])$, it does for $\mathrm{Sel}_{(\mathrm{lcm}(\mathfrak{l}_0, \mathfrak{l}))}(K, E[p])$, hence for $\mathrm{Sel}_{(\mathfrak{l})}(K, E[p])$.
\end{proof}

The next proposition states that non-triviality of a bipartite Kolyvagin system $Z$ is equivalent to the converse of Lemma \ref{lem-selsmall}.

\begin{prop}\label{nontriv-conv}
Let $Z$ be a bipartite Kolyvagin system for $E[p]$. The following assertions are equivalent.
\begin{enumerate}
\item $Z \neq Z_0$.
\item For every $\mathfrak{l} \in \mathcal{A}$,
\begin{equation*}
\mathrm{Sel}_{(\mathfrak{l})}(K, E[p])=0 \Rightarrow \mathfrak{l} \in \mathcal{A}^+ \text{ and } z_\mathfrak{l} \neq 0.
\end{equation*}
\item For every $\mathfrak{l} \in \mathcal{A}$,
\begin{equation*}
\mathrm{Sel}_{(\mathfrak{l})}(K, E[p])\simeq \Z/p\Z \Rightarrow \mathfrak{l} \in \mathcal{A}^- \text{ and } z_\mathfrak{l} \neq 0.
\end{equation*}
\end{enumerate}
\end{prop}
\begin{proof}\leavevmode
\begin{description}
\item[$(1)\Rightarrow (2)$] The argument below - which actually shows that $(1)$ implies $(2)$ and $(3)$ - is a straightforward adaptation of the one in the proof of \cite[Theorem 3.3.8]{swe22}; we include the details for completeness.

\textit{Step 0.} By Proposition \ref{par-prop} if $\mathfrak{l} \in \mathcal{A}$ and $\mathrm{Sel}_{(\mathfrak{l})}(K, E[p])=0$ (resp. $\mathrm{Sel}_{(\mathfrak{l})}(K, E[p])\simeq \Z/p\Z$) then $\mathfrak{l}\in \mathcal{A}^+$ (resp. $\mathfrak{l}\in \mathcal{A}^-$). Let $Z^\heartsuit\subset Z$ be the set of elements indexed by $\mathfrak{l} \in \mathcal{A}$ such that $\dim(\mathrm{Sel}_{(\mathfrak{l})}(K, E[p]))\leq 1$. We will show that every $z_\mathfrak{l} \in Z^\heartsuit$ is non zero.

\textit{Step 1.} Let $\mathfrak{l}\in \mathcal{A}^-$ and $\ell \nmid \mathfrak{l}$ be an admissible prime. If $z_\mathfrak{l}, z_{\mathfrak{l}\ell} \in Z^\heartsuit$ then
\begin{equation*}
z_\mathfrak{l}\neq 0 \Leftrightarrow z_{\mathfrak{l}\ell}\neq 0.
\end{equation*}
Indeed, if $z_\mathfrak{l}\neq 0$ then it generates $\mathrm{Sel}_{(\mathfrak{l})}(K, E[p])$; furthermore, by \eqref{chg-par} the dimension of $\mathrm{Sel}_{(\mathfrak{l}\ell)}(K, E[p])$ is even, hence it is zero. By \eqref{sel-1} we must have $\mathrm{loc}_\ell(z_\mathfrak{l})\neq 0$, hence $z_{\mathfrak{l}\ell}\neq 0$ by (RL2). Conversely, if $z_{\mathfrak{l}\ell}\neq 0$ then $z_\mathfrak{l}\neq 0$ by (RL2).

\textit{Step 2.} Let $\mathfrak{l}\in \mathcal{A}^+$ and $\ell \nmid \mathfrak{l}$ be an admissible prime. If $z_\mathfrak{l}, z_{\mathfrak{l}\ell} \in Z^\heartsuit$ then
\begin{equation*}
z_\mathfrak{l}\neq 0 \Leftrightarrow z_{\mathfrak{l}\ell}\neq 0.
\end{equation*}
Indeed, if $z_\mathfrak{l}\neq 0$ then $z_{\mathfrak{l}\ell}\neq 0$ by (RL1). Conversely, if $z_{\mathfrak{l}\ell}\neq 0$ then it generates $\mathrm{Sel}_{(\mathfrak{l}\ell)}(K, E[p])$; furthermore, $\mathrm{Sel}_{(\mathfrak{l})}(K, E[p])$ is trivial. By the discussion in \cref{sel-gps} we have $\mathrm{loc}_{\ell}(z_{\mathfrak{l}\ell}) \neq 0$, and (RL1) implies that $z_\mathfrak{l}\neq 0$.

\textit{Step 3.} By assumption there is $\mathfrak{l}_0\in \mathcal{A}$ such that $z_{\mathfrak{l}_0}\neq 0$ (such a $z_{\mathfrak{l}_0}$ automatically belongs to $Z^\heartsuit$ by Lemma \ref{lem-selsmall}). Let $z_\mathfrak{l}$ be an element in $Z^\heartsuit$. We will prove that there is a ``path'' in $Z^\heartsuit$ from $z_{\mathfrak{l}_0}$ to $z_\mathfrak{l}$ such that the index of each element in the path is obtained from the index of the previous element adding or removing an admissible prime. By the previous steps, this will imply that $z_\mathfrak{l}\neq 0$, as we wanted to prove.

Note that there is $\mathfrak{l}' \in \mathcal{A}$ multiple of $\mathfrak{l}$ and $\mathfrak{l}_0$ such that $z_{\mathfrak{l}'}\in Z^\heartsuit$ (using \cref{bip-set}(2) and \eqref{sel-1}, add admissible primes to $\mathrm{lcm}(\mathfrak{l}, \mathfrak{l}_0)$ if needed, until the dimension of the Selmer group drops at most to one). Therefore it suffices to show that, given any $z_{\mathfrak{l_1}}, z_{\mathfrak{l_2}} \in Z^\heartsuit$ with $\mathfrak{l}_1\mid \mathfrak{l}_2$ and $\mathfrak{l}_1\neq \mathfrak{l}_2$, we can connect $z_{\mathfrak{l_1}}$ to $ z_{\mathfrak{l_2}}$ via a path as above. We prove this by induction on the number $d$ of prime factors of $\mathfrak{l}_3=\mathfrak{l}_2/\mathfrak{l}_1$; the statement for $d=1$ is true by definition. Assume that $d>1$; for every $\ell \mid \mathfrak{l}_3$ we know by induction that, if $z_{\mathfrak{l_2}/\ell}\in Z^\heartsuit$, then $z_{\mathfrak{l}_1}$ and $z_{\mathfrak{l_2}/\ell}$ are connected (via a path as above), hence we are done. If $z_{\mathfrak{l_2}/\ell}$ is not in $Z^\heartsuit$ for any $\ell \mid \mathfrak{l}_3$, by the discussion in \cref{sel-gps} we must have $\dim(\mathrm{Sel}_{(\mathfrak{l_2}/\ell)}(K, E[p]))=2$ and $\mathrm{loc}_{\ell}(\mathrm{Sel}_{(\mathfrak{l}_2)}(K, E[p]))=0$ for every $\ell \mid \mathfrak{l}_3$. Therefore we find that $\mathrm{Sel}_{(\mathfrak{l}_2)}(K, E[p])\subset \mathrm{Sel}_{(\mathfrak{l}_1)}(K, E[p])$. As $\mathrm{Sel}_{(\mathfrak{l}_2)}(K, E[p])$ is one-dimensional, the inclusion is an equality. Choose an admissible prime $\ell'\nmid \mathfrak{l}_2$ such that $\mathrm{loc}_{\ell'}(\mathrm{Sel}_{(\mathfrak{l}_1)}(K, E[p]))\neq 0$; then $\mathrm{Sel}_{(\mathfrak{l}_1\ell')}(K, E[p]))=\mathrm{Sel}_{(\mathfrak{l}_2\ell')}(K, E[p]))=0$, so $\mathrm{Sel}_{(\mathfrak{l}_2\ell'/\ell)}(K, E[p]))\simeq \Z/p\Z$ for every $\ell \mid \mathfrak{l}_3$. Now we can connect, for any $\ell \mid \mathfrak{l}_3$,
\begin{equation*}
z_{\mathfrak{l}_1}\rightsquigarrow	z_{\mathfrak{l}_1\ell '} \rightsquigarrow z_{\mathfrak{l}_2\ell '/\ell} \rightsquigarrow z_{\mathfrak{l}_2\ell '}\rightsquigarrow z_{\mathfrak{l}_2}.
\end{equation*}
\item[$(2)\Rightarrow (3)$] If $\mathrm{Sel}_{(\mathfrak{l})}(K, E[p])\simeq \Z/p\Z$ then by \cref{bip-set}(2) and \eqref{sel-1} there is $\ell \nmid \mathfrak{l}$ admissible such that $\mathrm{Sel}_{(\mathfrak{l}\ell)}(K, E[p])=0$, hence $\mathfrak{l}\ell \in \mathcal{A}^+$ and $z_{\mathfrak{l}\ell}\neq 0$. Therefore $\mathfrak{l} \in \mathcal{A}^-$ and $z_\mathfrak{l}\neq 0$ by (RL2).
\item[$(3)\Rightarrow (1)$] Using \cref{bip-set}(2) and \eqref{sel-1} repeatedly we see that there is $\mathfrak{l}\in \mathcal{A}$ such that $\mathrm{Sel}_{(\mathfrak{l})}(K, E[p])=0$. For every $\ell \in \mathcal{A}$ not dividing $\mathfrak{l}$ we have $\mathrm{Sel}_{(\mathfrak{l}\ell)}(K, E[p])\simeq \Z/p\Z$ (by \eqref{chg-par}), hence $z_{\mathfrak{l}\ell}\neq 0$.
\end{description}
\end{proof}

\subsubsection{}\label{bipvspar} To sum up, the linear-algebraic gymnastics done in this section gives us the following information.

\begin{enumerate}[label=(\roman*)]
\item The existence of a non-trivial bipartite Kolyvagin system for $E[p]$ implies that $\mathrm{dim}(\mathrm{Sel}(K, E[p]))\equiv |\{q \mid N^-\}|+1 \pmod 2$ (this is the conclusion of Proposition \ref{par-prop} for $\mathfrak{l}=1$).
\item Conversely, suppose that $\mathrm{dim}(\mathrm{Sel}(K, E[p]))\equiv |\{q \mid N^-\}|+1 \pmod 2$; then the congruence \eqref{par-sel} holds for every $\mathfrak{l}\in \mathcal{A}$ (cf. the proof of Proposition \ref{par-prop}). For every $\mathfrak{l} \in \mathcal{A}$ such that $\mathrm{Sel}_{(\mathfrak{l})}(K, E[p])=0$ (resp. $\mathrm{Sel}_{(\mathfrak{l})}(K, E[p])\simeq \Z/p\Z$) pick a non-zero element $z_{\mathfrak{l}}$ in $\Z/p\Z$ (resp. in $\mathrm{Sel}_{(\mathfrak{l})}(K, E[p])$). We will now check that these elements form a bipartite Kolyvagin system $Z$; by Proposition \ref{nontriv-conv}, any other non-trivial bipartite Kolyvagin system for $E[p]$ is obtained multiplying the elements of $Z$ by elements of $(\Z/p\Z)^\times$.\\
Let $\mathfrak{l} \in \mathcal{A}$ and let $\ell \nmid \mathfrak{l}$ be an admissible prime. If $\mathrm{dim}(\mathrm{Sel}_{(\mathfrak{l})}(K, E[p]))\geq 3$ then $\mathrm{dim}(\mathrm{Sel}_{(\mathfrak{l}\ell)}(K, E[p]))\geq 2$ and both $z_{\mathfrak{l}}$ and $z_{\mathfrak{l}\ell}$ are zero. If $\mathrm{dim}(\mathrm{Sel}_{(\mathfrak{l})}(K, E[p]))=2$ then $z_\mathfrak{l}=0$; furthermore, by \eqref{chg-par}, the dimension of $\mathrm{Sel}_{(\mathfrak{l}\ell)}(K, E[p])$ is 1 or 3. In the latter case $z_{\mathfrak{l}\ell}=0$; in the former case, by \eqref{sel-1} we must have $\mathrm{loc}_\ell(\mathrm{Sel}_{(\mathfrak{l}\ell)}(K, E[p]))=0$, hence $\mathrm{loc}_\ell(z_{\mathfrak{l}\ell})=0$. If $\mathrm{dim}(\mathrm{Sel}_{(\mathfrak{l})}(K, E[p]))=1$ then $\mathrm{Sel}_{(\mathfrak{l}\ell)}(K, E[p])$ has dimension 0 or 2, and $\mathrm{Sel}_{(\mathfrak{l}\ell)}(K, E[p])=0$ if and only if $\mathrm{loc}_\ell(\mathrm{Sel}_{(\mathfrak{l})}(K, E[p]))\neq 0$, hence $\mathrm{loc}_\ell(z_\mathfrak{l})\neq 0 \Leftrightarrow z_{\mathfrak{l}\ell}\neq 0$. Finally, if $\mathrm{Sel}_{(\mathfrak{l})}(K, E[p])=0$ then $\mathrm{Sel}_{(\mathfrak{l}\ell)}(K, E[p])$ is one-dimensional and $\mathrm{loc}_\ell(\mathrm{Sel}_{(\mathfrak{l}\ell)}(K, E[p]))$ is non-zero, hence $\mathrm{loc}_\ell(z_{\mathfrak{l}\ell})\neq 0$.
\end{enumerate}

Note that $\mathrm{Sel}(K, E[p])=\mathrm{Sel}(K, E[p^\infty])[p]$ (cf. \cite[Proposition 5.8]{tam21}), so the dimension of $\mathrm{Sel}(K, E[p])$ and the $\Z_p$-corank of $\mathrm{Sel}(K, E[p^\infty])$ have the same parity (recall that, thanks to Cassels--Tate, we know that the cardinality of the torsion part of $\Sh(E/K)[p^\infty]$ is a square). Therefore, the congruence in (i) is the parity conjecture for $\mathrm{Sel}(K, E[p^\infty])$. Summing up, we learn that the existence of a non-trivial $Z$ is equivalent to the $p$-parity conjecture (cf. also \cite[Proposition 8.8]{tam21}, \cite[\S 9.2]{zha14}). The conjecture is already known - in bigger generality - by work of Nekov{\'a}{\v{r}} \cite{nek01}, \cite{nek06}, which exploits (in a different way) the non-triviality of a suitable Euler system \cite{cv07}.

A construction of an explicit, non-trivial $Z$, consisting of elements $z_\mathfrak{l}$ related to special values of $L$-functions of modular forms, yields further information on the Birch and Swinnerton-Dyer conjecture, as we will recall in \cref{bdks} below.

\subsection{A basis of $\mathrm{Sel}(K, E[p])$} Let $Z$ be a non-trivial bipartite Kolyvagin system for $E[p]$. By Proposition \ref{nontriv-conv}, if  $\mathrm{Sel}(K, E[p])$ is one-dimensional, then it is generated by the class $z_1$. The following theorem generalizes this observation, without imposing any restriction on the dimension of $\mathrm{Sel}(K, E[p])$.

\begin{prop}\label{gen-sel-abs}
Let $Z$ be a non-trivial bipartite Kolyvagin system for $E[p]$, and let $r=\mathrm{dim}(\mathrm{Sel}(K, E[p]))$. There exists $\mathfrak{l}=\ell_1\ell_2\cdots \ell_r \in \mathcal{A}$ such that $\{z_{\mathfrak{l/\ell_i}}, 1 \leq i \leq r\}$ is a basis of $\mathrm{Sel}(K, E[p])$.
\end{prop}
\begin{proof}
Using \cref{bip-set}(2) and \eqref{sel-1}, choose distinct admissible primes $\ell_1, \ldots, \ell_r$ such that, letting $\mathfrak{l}=\ell_1\cdots\ell_r$, we have $\mathrm{Sel}_{(\mathfrak{l})}(K, E[p])=0$. By Proposition \ref{par-prop}, for every $1 \leq i \leq r$ we have $\mathfrak{l}/\ell_i \in \mathcal{A}^-$, hence $z_{\mathfrak{l}/\ell_i} \in \mathrm{Sel}_{(\mathfrak{l}/\ell_i)}(K, E[p])$. Fix $1 \leq i, j \leq r$ with $j \neq i$. By \eqref{nontr-sel} we have $\mathrm{Sel}_{(\mathfrak{l}/(\ell_i\ell_j))}(K, E[p])\neq 0$, hence $\mathrm{loc}_{\ell_j}(z_{\mathfrak{l}/\ell_i})=0$ by (RL1) and Lemma \ref{lem-selsmall}. Therefore, $z_{\mathfrak{l}/\ell_i}$ belongs to $\mathrm{Sel}(K, E[p])$. Finally, by (RL2) and Proposition \ref{nontriv-conv}, the image of $(z_{\mathfrak{l/\ell_i}})_{1 \leq i \leq r}$ via the map
\begin{equation*}
\mathrm{loc}_{\ell_1}\oplus \cdots \oplus \mathrm{loc}_{\ell_r}: \mathrm{Sel}(K, E[p])\rightarrow H^1_{ur}(K_{\ell_1}, E[p])\oplus \cdots \oplus H^1_{ur}(K_{\ell_r}, E[p])\simeq (\Z/p\Z)^r
\end{equation*}
is the canonical basis of $(\Z/p\Z)^r$ (up to multiplication by a unit on each factor).
\end{proof}

\subsection{Bertolini--Darmon's bipartite Kolyvagin system modulo $p$}\label{bdks} We will now work with Bertolini--Darmon's bipartite Kolyvagin system $Z^{BD}$ for $E[p]$, whose construction relies on congruences of modular forms.

\subsubsection{Level raising}\label{levrais} For every $\mathfrak{l} \in \mathcal{A}$, there is a cuspidal, new (normalized) eigenform $f_{\mathfrak{l}}\in S_2(\Gamma_0(N\mathfrak{l}))$ and a prime ideal $\mathfrak{p}_\mathfrak{l}$ of the ring $\mathcal{O}_{\mathfrak{l}}\subset \C$ generated by the Fourier coefficients of $f_{\mathfrak{l}}$ such that $\mathcal{O}_{\mathfrak{l}}/\mathfrak{p}_\mathfrak{l}\simeq \Z/p\Z$ and the residual Galois representation $\bar{\rho}_{f_\mathfrak{l}}: \mathrm{Gal}(\bar{\Q}/\Q)\rightarrow \mathrm{GL}_2(\mathcal{O}_{\mathfrak{l}}/\mathfrak{p}_\mathfrak{l})$ is isomorphic to $\bar{\rho}$. This follows from \cite[Theorem 1]{dt94}, as explained in the proof of \cite[Theorem 2.1]{zha14} (note that the isomorphism $\bar{\rho}_{f_\mathfrak{l}}\simeq \bar{\rho}$ forces $\mathcal{O}_{\mathfrak{l}}/\mathfrak{p}_\mathfrak{l}\simeq \Z/p\Z$). Let $\bar{\mathcal{O}}_\mathfrak{l}$ be the integral closure of $\mathcal{O}_\mathfrak{l}$, and fix a prime ideal $\bar{\mathfrak{p}}_\mathfrak{l}\subset \bar{\mathcal{O}}_\mathfrak{l}$ such that $\bar{\mathfrak{p}}_\mathfrak{l}\cap \mathcal{O}_{\mathfrak{l}}=\mathfrak{p}_{\mathfrak{l}}$.

\subsubsection{Assumptions}\label{ass-bd} In addition to the hypotheses in \cref{bip-set}, let us assume throughout the rest of this section that $N$ is squarefree and, if either $q \mid N^-$ and $q \equiv \pm 1 \pmod p$, or $q \mid N^+$, then $\bar{\rho}$ is ramified at $q$. In particular, on the one hand hypothesis CR in \cite{pw11} holds; on the other hand, the local conditions defining $\mathrm{Sel}(K, E[p])$ can be expressed purely in terms of the $\mathrm{Gal}(\bar{K}/K)$-module $E[p]$; in particular, the mod $\mathfrak{p}_{\mathfrak{l}}$ Selmer group of $f_{\mathfrak{l}}$ is isomorphic to $\mathrm{Sel}_{(\mathfrak{l})}(K, E[p])$ (cf. the proof of \cite[Theorem 5.2]{zha14}).

\subsubsection{Construction of $Z^{BD}$}\label{constr-bd} Under the above assumptions, recall that the bipartite Kolyvagin system $Z^{BD}=\{z_\mathfrak{l}^{BD}, \mathfrak{l} \in \mathcal{A}\}$ is constructed as follows. For every $\mathfrak{l}\in \mathcal{A}^+$ (resp. $\mathfrak{l}\in \mathcal{A}^-$) let $B_\mathfrak{l}$ be the quaternion algebra over $\Q$ ramified at $N^-\mathfrak{l}\infty$ (resp. $N^-\mathfrak{l}$), and let $\mathrm{Sh}_{N^+}(B_\mathfrak{l}^\times)$ be the Shimura set (resp. Shimura curve) of level $\Gamma_0(N^+)$ attached to $B_\mathfrak{l}^\times$. In both cases, let $CM(\mathcal{O}_K)$ be the set of points with $CM$ by $\mathcal{O}_K$.
\begin{enumerate}
\item For $\mathfrak{l}\in \mathcal{A}^+$, the vector space of functions $\mathrm{Sh}_{N^+}(B_\mathfrak{l}^\times) \rightarrow \bar{\mathcal{O}}_\mathfrak{l}/\bar{\mathfrak{p}}_\mathfrak{l}$ on which the Hecke algebra acts as on $f_\mathfrak{l}$ is non-zero (cf. \cite[\S 6.1]{zha14}), hence so is the space of $\mathbf{F}_p$-valued functions on $\mathrm{Sh}_{N^+}(B_\mathfrak{l}^\times)$ with the same property. Furthermore, this $\mathbf{F}_p$-vector space is one-dimensional (cf. \cite[Proposition 6.5, Theorem 6.2]{pw11}). Fixing a basis $\bar{f}_{\mathfrak{l}}^{B_{_\mathfrak{l}}}$, the element $z_\mathfrak{l}^{BD} \in \Z/p\Z$ is the sum $\sum_{Q \in CM(\mathcal{O}_K)}\bar{f}_{\mathfrak{l}}^{B_{\mathfrak{l}}}(Q)$.
\item For $\mathfrak{l}\in \mathcal{A}^-$, the class $z_\mathfrak{l}^{BD} \in \mathrm{Sel}_{(\mathfrak{l})}(K, E[p])$ is obtained from the divisor $\sum_{Q \in CM(\mathcal{O}_K)}Q$ on $\mathrm{Sh}_{N^+}(B_\mathfrak{l}^\times)$. We will explain this in more detail in \cref{const-zl}, where the precise construction will be needed.
\end{enumerate}
The proof that the above classes satisfy the relations yielding a bipartite Kolyvagin system is non-trivial; it rests on the structure of special fibers of Shimura curves - allowing to realize level raising and Jacquet--Langlands functoriality geometrically - as well as on Ihara's lemma and multiplicity one modulo $p$. See \cite[\S 8, 9]{bd05} and \cite{pw11}.

\subsubsection{Non triviality of $Z^{BD}$}\label{zbdntriv} For $\mathfrak{l}\in \mathcal{A}^+$, the element $z_\mathfrak{l}^{BD}$ non-zero if and only if the algebraic part of the special value $L^{alg}(f_\mathfrak{l}/K, 1)$, as defined in \cite[(6.5)]{zha14}, is not congruent to zero modulo $\bar{\mathfrak{p}}_\mathfrak{l}$ (this follows from Gross's formula, cf. \cite[Corollary 6.2]{zha14}). Now, we have the following results (in increasing order of difficulty):
\begin{enumerate}
\item the ($\bar{\mathfrak{p}}_\mathfrak{l}$-adic) valuation of the Tamagawa numbers $t_q(f_\mathfrak{l})$ at primes $q \mid N^+$ is zero by \cref{ass-bd};
\item the valuation of the ratio between the Gross period used to define $L^{alg}(f_\mathfrak{l}/K, 1)$ and the canonical period $\Omega^{can}_{f_\mathfrak{l}}$ equals the sum of the valuations of the Tamagawa numbers at primes dividing $N^-\mathfrak{l}$ \cite[Theorem 6.8]{pw11};
\item assume that either $E$ is ordinary at $p$, or $E$ is supersingular at $p$ and there is $q \mid N^+$ such that the restriction of $\bar{\rho}$ to $\mathrm{Gal}(\bar{\Q}_q/\Q_q)$ is a ramified extension of $\mu$ by $\mu\chi_{cycl}$, where $\mu$ is the non-trivial unramified quadratic character of $\mathrm{Gal}(\bar{\Q}_q/\Q_q)$. Then the following implication holds (Tamagawa numbers at places dividing $N^+$ are irrelevant in view of $(1)$):
\begin{equation}\label{modp-conv}
\mathrm{Sel}_{(\mathfrak{l})}(K, E[p])=0 \Rightarrow \frac{L(f_{\mathfrak{l}}/K, 1)}{\Omega^{can}_{f_\mathfrak{l}}\prod_{q \mid N^- \mathfrak{l}}t_q(f_\mathfrak{l})} \not \equiv 0 \pmod {\bar{\mathfrak{p}}_\mathfrak{l}}.
\end{equation}
In the ordinary case, this follows from work of Kato and Skinner--Urban (cf. \cite[Theorem 7.1]{zha14}, whose third hypothesis is satisfied under our assumptions by \cite[Theorem 1.1]{rib90}). In the supersingular case, \eqref{modp-conv} follows from the work of Fouquet--Wan \cite{fowa22} (Theorem 1.7, Corollary 1.9, and an analogue of Corollary 1.10 for the motives with $\mathrm{Frac}(\mathcal{O}_\mathfrak{l})$-coefficients attached to $f_\mathfrak{l}$ and to its twist by the quadratic character attached to $K$).

As a consequence of (1), (2) and (3), the equivalent conditions of Proposition \ref{nontriv-conv} are satisfied, hence $Z^{BD}\neq Z_0$. Indeed, if $\mathfrak{l}\in \mathcal{A}$ and $\mathrm{Sel}_{(\mathfrak{l})}(K, E[p])=0$ then $\mathfrak{l}\in \mathcal{A}^+$ (by the parity conjecture \cite[Theorem B]{nek13}, or by \cite[Theorem 7.1]{zha14} and \cite[Corollary 1.9]{fowa22}), and by the previous discussion we have $L^{alg}(f_\mathfrak{l}/K, 1)\not \equiv 0 \pmod {\bar{\mathfrak{p}}_\mathfrak{l}}$.
\end{enumerate}

\subsubsection{}\label{weak-bsd} As we have seen above, for $\mathfrak{l} \in \mathcal{A}^+$ the element $z_{\mathfrak{l}}^{BD}$ is related to the central value $L(f_{\mathfrak{l}}/K, 1)$; similarly, for $\mathfrak{l} \in \mathcal{A}^-$ the class $z_\mathfrak{l}^{BD}$ is related to $L'(f_{\mathfrak{l}}/K, 1)$ (by the Gross--Zagier formula for Shimura curves \cite{xzz13}). Thanks to these relations, the abstract results in the previous sections acquire the following arithmetic meaning.
\begin{enumerate}
\item Lemma \ref{lem-selsmall}(1) tells us that, if $\mathfrak{l}\in \mathcal{A}^+$ and $L^{alg}(f_{\mathfrak{l}}/K, 1) \not \equiv 0 \pmod{\bar{\mathfrak{p}}_\mathfrak{l}}$, then $\mathrm{Sel}_{(\mathfrak{l})}(K, E[p])=0$. In particular, if $L(E/K, 1)\neq 0$ then $E(K)$ is finite. A proof of this implication in much greater generality, using this circle of ideas, can be found in \cite{nek12}.

Similarly, it follows from Lemma \ref{lem-selsmall}(2) that, if $\mathrm{ord}_{s=1}L(E/K, s)=1$, then $E(K)$ has rank one, and $\Sh(E/K)[p]=0$ if $z_1^{BD}\neq 0$.
\item Proposition \ref{nontriv-conv} relates the non-triviality of $Z^{BD}$ to converses of the statements in the previous point, which in turn yield various ``converse theorems''. For instance, Proposition \ref{nontriv-conv}(3) implies that, if $E(K)$ has rank one and $\Sh(E/K)$ is finite, then $L'(E/K, 1)\neq 0$, a converse of the theorem of Gross--Zagier and Kolyvagin.

Of course, it would be of interest if one could prove that $Z^{BD}\neq Z_0$ directly, without appealing to the ``mod $p$ converse'' \eqref{modp-conv}.
\item The proof of the implication $(2)\Rightarrow (3)$ in Proposition \ref{nontriv-conv} is essentially an abstract formulation of the argument in \cite[\S 7.2]{zha14}, which is at the heart of Zhang's proof of the non-triviality of Kolyvagin's system of Heegner points. In a nutshell, Zhang reduces the latter to the non-triviality of $Z^{BD}$.
\item Finally, Proposition \ref{gen-sel-abs} yields Theorem \ref{bd-gensel} below.
\end{enumerate}

\begin{thm}\label{bd-gensel}
Let $E/\Q$ be an elliptic curve of squarefree conductor $N$, without complex multiplication. Let $K$ be an imaginary quadratic field such that the prime factors $q\mid N$ are unramified in $K$, and write $N=N^-N^+$, where $q \mid N^-$ (resp. $q \mid N^+$) if $q$ is inert (resp. split) in $K$. Let $p>3$ be a prime not dividing $N\mathrm{disc}(K/\Q)$, and let $r=\mathrm{dim}(\mathrm{Sel}(K, E[p]))$. Assume that $\bar{\rho}: \mathrm{Gal}(\bar{\Q}/\Q)\rightarrow \mathrm{Aut}_{\mathbf{F}_p}(E[p])$ has the following properties:
\begin{enumerate}
\item $\bar{\rho}$ is surjective;
\item if $q \mid N^-$ and $q \equiv \pm 1 \pmod p$ then $\bar{\rho}$ is ramified at $q$;
\item if $q \mid N^+$ then $\bar{\rho}$ is ramified at $q$; furthermore, if $E$ is supersingular at $p$, there is $q \mid N^+$ such that the restriction of $\bar{\rho}$ to $\mathrm{Gal}(\bar{\Q}_q/\Q_q)$ is a ramified extension of $\mu$ by $\mu\chi_{cycl}$, where $\mu$ is the non-trivial unramified quadratic character of $\mathrm{Gal}(\bar{\Q}_q/\Q_q)$.
\end{enumerate}
Then there exists $\mathfrak{l}=\ell_1\ell_2\cdots \ell_r \in \mathcal{A}$ such that $\{z_{\mathfrak{l/\ell_i}}^{BD}, 1 \leq i \leq r\}$ is a basis of $\mathrm{Sel}(K, E[p])$.
\end{thm}
\begin{proof}
The assumptions ensure that $Z^{BD}$ is defined and non trivial, cf. \cref{constr-bd} and \cref{zbdntriv}; hence, the theorem follows from Proposition \ref{gen-sel-abs}.
\end{proof}

\subsubsection{}\label{crit} A weaker version of the theorem was proved in \cite[Theorem 11.1]{zha14}.

Note that Theorem \ref{bd-gensel} is valid without any restriction on the rank of $E(K)$. If the rank is bigger than one, no general construction of a set of (topological) generators of $\mathrm{Sel}(K, T_p(E))$ is known; on the other hand, Theorem \ref{bd-gensel} tells us that $\mathrm{Sel}(K, E[p])$ is always generated by classes coming from $CM$ points on suitable Shimura curves. However, for $r>1$ the bases $\{z_{\mathfrak{l/\ell_i}}^{BD}, 1 \leq i \leq r\}$ of $\mathrm{Sel}(K, E[p])$ obtained through the theorem are very non-canonical, and their relation with the arithmetic properties of $E$ is mysterious. The Birch and Swinnerton-Dyer conjecture predicts that $\mathrm{ord}_{s=1}L(E/K, s)\leq r$, and equality holds if and only if all the Selmer classes $z_{\mathfrak{l/\ell_i}}^{BD}$ in the statement of the theorem come from $K$-points of $E$. However, our approach offers no way to understand which classes have this property, and it is unclear how to relate the elements $z_{\mathfrak{l/\ell_i}}^{BD}$ to (higher derivatives of) $L(E/K, s)$. For these reasons, the above theorem does not seem to have direct applications to the Birch and Swinnerton-Dyer conjecture. We will instead apply it in the next section to the study of visibility of Tate--Shafarevich classes.

\begin{rem}
In view of the aforementioned relations between the elements $z_\mathfrak{l}^{BD}$ and $L$-values, the reciprocity laws (RL1), (RL2) can be roughly interpreted as algebraic incarnations of the idea, suggested in \cite{maz79}, that congruences of modular forms should correspond to congruences between values of their $L$-functions. We will not attempt to list here the myriad of works related to this topic, but we would like to remind the reader of \cite{joc94}, where Jochnowitz suggests to look at ``examples where the forms are congruent, but the root numbers are opposite [...] This should happen, for example, any time you had a newform of square-free level that was congruent mod $p$ to a form of lower level. It would be interesting to work out some concrete examples, and see if the derivatives come in as we predicted.'' \cite[p. 255]{joc94}. The reciprocity laws for $Z^{BD}$ are a manifestation of the fact that Jochnowitz's visionary prediction was essentially correct.
\end{rem}

\subsubsection{Primitivity of Bertolini--Darmon's bipartite Kolyvagin system}\label{bd-prim} Let us collect further consequences of the non-triviality of $Z^{BD}$, which will not be needed in the sequel of this article but might have independent interest. The construction of $Z^{BD}$ recalled above can be extended (working with modular forms modulo $p^k$ congruent to $f$) to obtain elements in $\Z/p^k\Z$ and $\mathrm{Sel}_{(\mathfrak{l})}(K, E[p^k])$ for $k \geq 1$. See \cite{bd05}, \cite{pw11}; the relevant abstract theory was developed in \cite{how06}. The condition in Proposition \ref{nontriv-conv}(1) can be seen as an analogue in this setting of the notion of primitivity of usual Kolyvagin systems \cite[Definition 4.5.5]{mr04}. Analogously to \cite[Theorem 4.5.6]{mr04}, this property implies that the inequalities of \cite[Theorem 1.5]{tam21} are equalities (cf. also \cite[Theorem 2.5.1]{how06}). For $E/K$ and $p$ as in Theorem \ref{bd-gensel}, this yields the $p$-part of the Birch and Swinnerton-Dyer special value formula in analytic rank zero and one (in the rank one case, the formula was obtained in \cite[Theorem 10.2]{zha14} as a byproduct of the primitivity of the Kolyvagin system of Heegner points); this idea was developed in \cite{bbv16}, \cite{blv}.

Non-triviality of $Z^{BD}$ also implies more precise converse theorems than the one mentioned in \cref{weak-bsd}(2). For instance, let us sketch an argument proving the following implication, assuming that $Z^{BD}$ is non trivial (which is the case under the assumptions of Theorem \ref{bd-gensel}):
\begin{equation}\label{pconv}
\mathrm{Sel}(K, E[p^\infty]) \text{ has } \Z_p\text{-corank one } \Rightarrow L'(E/K, 1)\neq 0.
\end{equation}

Write $\mathrm{Sel}(K, E[p^\infty])=\Q_p/\Z_p\oplus S$ with $S=T\oplus T$ finite. By the Gross--Zagier formula for Shimura curves, it suffices to prove that the class $z_{1, k}^{BD} \in \mathrm{Sel}(K, E[p^{k}])$ coming from the divisor considered in \cref{constr-bd}(2) for $\mathfrak{l}=1$ is non-zero for some $k\geq 1$. Let us fix $k$ larger than twice the length $l(T)$ of $T$. In the rest of this subsection, we will work with Bertolini--Darmon's bipartite Kolyvagin system modulo $p^{k}$, whose elements will be denoted by $z_{\mathfrak{l}, k}^{BD}$. As in \cite{how06} and \cite{tam21}, the elements $z_{\mathfrak{l}, k}^{BD}$ are obtained as reductions modulo $p^k$ of the analogous elements attached to $\mathfrak{l}\in \mathcal{A}_{2k}$ squarefree product of primes $\ell \nmid pN$ inert in $K$ such that $\ell \not \equiv \pm 1 \pmod p$ and $a_\ell \equiv \pm(\ell+1) \pmod{p^{2k}}$. For every $\mathfrak{l}\in \mathcal{A}_{2k}$ with an even number of prime factors we have an isomorphism $\mathrm{Sel}_{(\mathfrak{l})}(K, E[p^k])\simeq \Z/p^k\Z\oplus T_\mathfrak{l}\oplus T_\mathfrak{l}$ (cf. \cite[Proposition 2.2.7, Corollary 2.2.10]{how06}). We will prove by induction on $l(T_{\mathfrak{l}})$ that, for such $\mathfrak{l}$, if $l(T_{\mathfrak{l}})\leq l(T)$ then $z_{\mathfrak{l}, k}^{BD}\neq 0$. 

Choose (using \cite[Theorem 3.2]{bd05}) $\ell_1 \in \mathcal{A}_{2k}$ not dividing $\mathfrak{l}$ such that $\mathrm{loc}_{\ell_1}(\mathrm{Sel}_{(\mathfrak{l})}(K, E[p^{k}]))=\Z/p^{k}\Z$, so that $\mathrm{Sel}_{(\mathfrak{l}\ell_1)}(K, E[p^{k}])\simeq T_\mathfrak{l} \oplus T_\mathfrak{l}$ by global duality. If $T_\mathfrak{l}=0$ then $z_{\mathfrak{l}\ell_1, k}^{BD}\neq 0$, so $z_{\mathfrak{l}, k}^{BD}\neq 0$ by the second reciprocity law. Otherwise, let $p^t$ be the exponent of $T_\mathfrak{l}$ and write $T_{\mathfrak{l}}=\Z/p^t\Z\oplus T'$. Choose $\ell_2\in \mathcal{A}_{2k}$ not dividing $\mathfrak{l}\ell_1$ such that $\mathrm{loc}_{\ell_2}$ is injective on $\Z/p^t\Z$, so that $\mathrm{loc}_{\ell_2}(\mathrm{Sel}_{(\mathfrak{l}\ell_1)}(K, E[p^{k}]))\simeq \Z/p^t\Z$. The analogue modulo $p^{k}$ of the diagram in \cref{sel-gps} (cf. \cite[Proposition 2.2.9]{how06}) implies that $\mathrm{Sel}_{(\mathfrak{l}\ell_1\ell_2)}(K, E[p^{k}])\simeq \Z/p^{k}\Z\oplus T'\oplus T'$. Hence $T'\simeq T_{\mathfrak{l}\ell_1\ell_2}$, and $z_{\mathfrak{l}\ell_1\ell_2, k}^{BD}\neq 0$ by the inductive hypothesis.

We will show that $\mathrm{loc}_{\ell_2}(z_{\mathfrak{l}\ell_1\ell_2, k}^{BD})\neq 0$, which implies that $z_{\mathfrak{l}\ell_1, k}^{BD}\neq 0$ and $z_{\mathfrak{l}, k}^{BD}\neq 0$ by the reciprocity laws, concluding the proof. We argue by contradiction. Let $e\geq 0$ be the greatest integer such that $z_{\mathfrak{l}\ell_1\ell_2, k}^{BD}\in p^e\mathrm{Sel}_{(\mathfrak{l}\ell_1\ell_2)}(K, E[p^k])$. The proof of \cite[Theorem 8.3]{tam21} shows (under our assumption that $Z^{BD}\neq Z_0$) that $e$ equals $l(T')$; furthermore, $z_{\mathfrak{l}\ell_1\ell_2, k}^{BD}$ is contained in a submodule of $\mathrm{Sel}_{(\mathfrak{l}\ell_1\ell_2)}(K, E[p^k])$ isomorphic to $\Z/p^k\Z$ (cf. \cite[Lemma 3.3.6]{how06}, \cite[Proposition 7.9]{tam21}). It follows that, if $\mathrm{loc}_{\ell_2}(z_{\mathfrak{l}\ell_1\ell_2, k}^{BD})=0$, then $\mathrm{loc}_{\ell_2}(\mathrm{Sel}_{(\mathfrak{l}\ell_1\ell_2)}(K, E[p^k]))$ is killed by $l(T')$, hence by $l(T_{\mathfrak{l}})$. This contradicts global duality, because $\mathrm{loc}_{\ell_2}(\mathrm{Sel}_{(\mathfrak{l}\ell_1)}(K, E[p^k]))$ is also killed by $l(T_{\mathfrak{l}})$ and $2l(T_{\mathfrak{l}})<k$.

In the ordinary case, the implication \eqref{pconv} was obtained, more indirectly, in \cite[Theorem 1.3]{zha14} as a consequence of the non-triviality of the Kolyvagin system of Heegner points and of Kolyvagin's structure theorem for the Selmer group \cite{kol91}. In the supersingular case, \eqref{pconv} was proved via different methods (and under different assumptions) in \cite[Theorem B]{cw}.

\subsubsection{Addendum to \cite{tam21}} In this paragraph, let $K/F$ be a quadratic $CM$ extension of a totally real field. As mentioned above, in \cite{tam21} a version of Bertolini--Darmon's bipartite Kolyvagin system is used to give upper bounds on the size of Selmer groups of Hilbert modular forms $f$ of parallel weight two and analytic rank at most one over $K$. For modular elliptic curves $E$ over $F$, these bounds are a step towards the $p$-part of the Birch and Swinnerton-Dyer special value formula \eqref{svf} for $E/K$. For instance, suppose that $L(E/K, 1)\neq 0$. As discussed in \cite[Remarks 5.5, 5.6]{tam21}, if one knows that the relevant inequality in \cite[Theorem 1.5]{tam21} is an equality, a comparison between the period $P(E/K)$ and the one appearing in the special value formula of \cite[Theorem 7.1]{zha04} is required to deduce the $p$-part of \eqref{svf}. This comparison is known under suitable assumptions for $F=\Q$ \cite[Theorem 6.8]{pw11}, but we would like to clarify that it appears to be much deeper for general $F$. For instance, for $F$ of even degree and $E$ with good reduction everywhere one needs to show that the quotient of the period in \cite[Theorem 7.1]{zha04} by the product of the archimedean periods of $E/K$ is a $p$-adic unit. This is closely related to conjectures of Ichino--Prasanna \cite[Conjectures A, C]{ip21} (whose rational, Hodge-theoretic counterpart is Oda's conjecture \cite[p. xii]{od82}).

\begin{rem}\label{termin}
We conclude this section with a brief remark on our choice of notation and terminology. We denoted bipartite Kolyvagin systems by $Z$ because of the relation with zeta elements which is apparent in the example of $Z^{BD}$. Furthermore, we chose the name ``bipartite \emph{Kolyvagin} system'' instead of the more common ``bipartite \emph{Euler} system'' because $Z^{BD}$ resembles more a classical Kolyvagin system than an Euler system. First of all, as already pointed out in the introduction of \cite{bd05}, there are no Kolyvagin derivatives in the definition of $Z^{BD}$: ramification of the classes $z_{\mathfrak{l}}^{BD}\in H^1(K, E[p])$ at places dividing $\mathfrak{l}$ comes from bad reduction of the Shimura curves used to construct $z_{\mathfrak{l}}^{BD}$, rather than from ramification of suitable field extensions of $K$ as in Kolyvagin's construction of the Euler system of Heegner points. In addition, let us point out that the structure of $\mathrm{Sel}(K, E[p^\infty])$ can be described in terms of invariants attached to systems of elements (modulo $p^k$ for arbitrary $k$) satisfying generalizations of (RL1) and (RL2), as in \cite[\S 4]{ki24}. This is akin to the information one can obtain from Kolyvagin systems, cf. for instance \cite[Proposition 4.5.8]{mr04}.

Note that the fact that $Z^{BD}$ is not constructed as a ``Kolyvagin derivative of an Euler system'' appears to be crucial for several purposes. Firstly, it ensures that all the classes in $Z^{BD}$ are directly related to special values of $L$-functions, which in turn allows to establish the non-triviality of $Z^{BD}$ (a similar remark can be found in \cite[\S 1]{how06}). Secondly, the fact that all the cohomology classes in $Z^{BD}$ directly come from points will be essential for our applications to visibility below.

The above discussion leads us to wonder whether there are other ways to construct directly objects behaving like Kolyvagin systems.
\end{rem}

\section{Background on visibility}\label{back-vis}

\subsection{Visibility} Let $K$ be a number field, $E/K$ an elliptic curve and $\iota: E \rightarrow A$ an embedding of $E$ into an abelian variety $A$ (over $K$). A class $c \in H^1(K, E)$ is visible in $A$ (via $\iota$) if $\iota(c)=0\in H^1(K, A)$. In other words, denoting by $\gamma: A \rightarrow B=A/E$ the quotient map, the class $c$ is visible in $A$ if it is the image of a point $P \in B(K)$ via the connecting map $B(K)\rightarrow H^1(K, E)$. Concretely, this means that $c$ corresponds to the $E$-torsor $\gamma^{-1}(P)\subset A$.

Let us start by recalling some basic facts about visibility.

\begin{lem}\label{lem-basic-vis}
For every $c\in H^1(K, E)$, the following assertions hold true.
\begin{enumerate}
\item Let $L/K$ be a finite extension such that $\mathrm{res}_{K}^L(c)=0 \in H^1(L, E)$. The class $c$ is visible in an abelian variety of dimension $[L: K]$.
\item Let $m \geq 1$ be an integer. If $c \in \Sh(E/K)[m]$ then $c$ is visible in an abelian variety of dimension at most $m$.
\end{enumerate}
\end{lem}

\begin{proof} \leavevmode
\begin{enumerate}
\item The class $c$ is visible in $A=\mathrm{Res}_{L/K}E_L$, as explained in the proof of \cite[Proposition 1.3]{ast02}.
\item This is \cite[Proposition 2.4]{ast02}. The point is that $c$ is in the image of $\mathrm{Sel}(K, E[m])$, whose elements give rise to $E$-torsors with an effective divisor of degree $m$, as observed by Cassels (cf. the proof of \cite[Theorem 1.3]{cas62}).
\end{enumerate}
\end{proof}

\subsection{Visibility and congruences}\label{vis-cong} Let $\iota : E \rightarrow A$ be an embedding. By Poincaré's reducibility theorem, we can choose an abelian subvariety $E^\perp\subset A$ such that there is a short exact sequence
\begin{equation*}
0 \rightarrow E \cap E^\perp \rightarrow E \times E^\perp \rightarrow A \rightarrow 0
\end{equation*}
where $E \cap E^\perp$ is finite over $K$ and is embedded anti-diagonally in $E \times E^\perp$. It follows that, if $c \in H^1(K, E)$ is visible in $A$, then $(c, 0)$ is in the image of $H^1(K, E \cap E^\perp)$ (in particular, $c$ is killed by the exponent of $E \cap E^\perp$). Reversing the argument yields a classical construction which allows to produce visible classes (cf. \cite[\S 3]{cm00}); we will make use of a variation of it in the next section. Let us illustrate the construction first in the simplest example: suppose that $F$ is an elliptic curve, $p$ is a prime and $E[p]\simeq F[p]$ is an isomorphism. Let $A=(E\times F)/E[p]$, where the embedding is anti-diagonal; the following diagram is commutative:
\begin{center}
\begin{tikzcd}
E[p] \arrow[r] \arrow[d]
& F \arrow[d] \\
E \arrow[r]
& A.
\end{tikzcd}
\end{center}
Let $P \in F(K)$, and let $c \in H^1(K, F[p])$ be its image via the connecting map $F(K)/pF(K)\rightarrow H^1(K, F[p])$. Then the image of  $c$ in $H^1(K, E)$ is visible in $A$. In particular, if $c$ belongs to $\mathrm{Sel}(K, E[p])$, then it yields a visible class in $\Sh(E/K)[p]$.

\subsubsection{Visibility of $\Sh(E/K)[3]$ in abelian surfaces}\label{maz-3} A first non-trivial application of the above construction is Mazur's theorem \cite{maz99} asserting that every class in $\Sh(E/K)[3]$ is visible in an abelian surface (note that Lemma \ref{lem-basic-vis} only guarantees that such a class is visible in an abelian threefold). Mazur shows that, for every $c \in \mathrm{Sel}(K, E[3])$, there is an elliptic curve $F$ with an isomorphism $F[3]\simeq E[3]$ such that $c$ goes to zero in $H^1(K, F)$. To do so, Mazur's idea in a nutshell is to consider the moduli space $\mathcal{S}(E, c)$ of elliptic curves $F$ together with an isomorphism $E[3]\simeq F[3]$, and a section of an $F$-torsor representing the image of $c$ in $H^1(K, F)$. As $c \in \mathrm{Sel}(K, E[3])$, the variety $\mathcal{S}(E, c)$ has a point over every completion of $K$, and one needs to show that it has a $K$-point. The key input is the fact that $\mathcal{S}(E, c)/\bar{K}$ is birationally equivalent to $\mathbf{P}^2$, which allows to apply the local-global principle (to a suitable Brauer--Severi variety constructed from $\mathcal{S}(E, c)$).

\subsubsection{} Note that only for small $n$ we can have many elliptic curves $F/K$ such that $F[n]\simeq E[n]$ (the moduli space of such $F$ is a form of the modular curve of full level $n$, whose genus increases); this suggests that classes in $\Sh(E/K)[n]$ should not, in general, be visible in abelian surfaces for $n$ large. Explicit examples for $n=6, 7$ are given in \cite{fis14}.

\subsection{Visibility in modular abelian varieties}\label{vis-modav} Let us now restrict to $K=\Q$. Let $E/\Q$ be an elliptic curve of conductor $N$; fix a modular parametrization $\pi: X_0(N) \rightarrow E$ and assume that the induced map $E\rightarrow J_0(N)=\mathrm{Jac}(X_0(N))$ is injective (such an $(E, \pi)$ is called optimal). With the notation of \cref{vis-cong}, we can take $E^\perp$ to be the kernel of the map $J_0(N)\rightarrow E$ induced by $\pi$, and we deduce that classes $c\in H^1(K, E)$ which are visible in $J_0(N)$ are killed by the modular degree, i.e. the degree of $\pi$ (cf. \cite[p. 20]{cm00}). It is not always the case that $\Sh(E/\Q)$ is killed by the modular degree, hence Tate--Shafarevich classes are not visible in $J_0(N)$ in general (cf. \cite[\S 4]{cm00}).

\subsubsection{Visibility at higher level} It is natural to wonder whether classes in $\Sh(E/\Q)$ which are not visible in $J_0(N)$ become visible in $J_0(M)$ for some multiple $M$ of $N$ (the analogous phenomenon of capitulation of ideal classes of abelian number fields in cyclotomic fields is studied in \cite{sw10}). We will be interested in a variation of this question, explained in the next paragraph. An example of visibility of Tate--Shafarevich classes at higher level is given in \cite[\S 4.2]{ast02}, where the authors exhibit an elliptic curve of conductor 5389  with Tate--Shafarevich classes invisible in $J_0(5389)$, but visible in $J_0(7\cdot 5389)$; as explained before Proposition 4.2 of \emph{loc. cit.}, the construction is related to level raising.

\subsubsection{Modularity} As in \cite[\S 2]{js07}, by a $J_0$-modular abelian variety we mean an abelian variety over $\Q$ which is a quotient of $J_0(M)$ for some $M \geq 1$.  We say that a class $c\in \Sh(E/\Q)$ is modular if there is an embedding $\iota: E \rightarrow A$ of $E$ into a $J_0$-modular abelian variety $A$ such that $c$ is visible in $A$. The following is a special case of \cite[Conjecture 7.1.1]{js07}.

\begin{conj}(cf. \cite[Conjecture 7.1.1]{js07})\label{mod-sha-conj}
Every class in $\Sh(E/\Q)$ is modular.
\end{conj}

\begin{rem}
As pointed out in \cite[Remarks 7.1.2, 7.1.3]{js07}, the following more precise problems would also be worth investigating; they might offer a way to relate the previous conjecture to the conjectural finiteness of $\Sh(E/\Q)$ (note that classes in $\Sh(E/\Q)$ which are visible via a fixed $\iota: E \rightarrow A$ form a finite group by the Mordell--Weil theorem).
\begin{enumerate}
\item Study the (minimal) levels of the $J_0$-modular abelian varieties where the classes in $\Sh(E/\Q)$ become visible.
\item Given a prime $\ell \nmid N$, the two degeneracy maps $X_0(N\ell)\rightarrow X_0(N)$ induce maps $J_0(N)\rightarrow J_0(N\ell)$; repeating the argument we get maps $J_0(N)\rightarrow J_0(N\ell_1\cdots \ell_i)$ for every product of distinct primes not dividing $N$. For $E\subset J_0(N)$, determine which classes in $\Sh(E/\Q)$ are visible in $J_0(N\ell_1\cdots \ell_i)$ via (linear combinations of) such maps.
\end{enumerate}
\end{rem}

\subsubsection{Known results} The class $0\in \Sh(E/\Q)$ is modular, because $E$ is modular. Besides this, in \cite{js07} the authors give the following theoretical evidence supporting Conjecture \ref{mod-sha-conj}.
\begin{enumerate}
\item Classes in $\Sh(E/\Q)$ of order $2$ or $3$ are modular \cite[Proposition 7.2.1]{js07}, because they are visible in an abelian surface (by Lemma \ref{lem-basic-vis} and \cref{maz-3}).
\item By \cite[Proposition 7.2.2]{js07}, classes in $\Sh(E/\Q)$ which split over an abelian extension of $\Q$ are visible in quotients of $\mathrm{Jac}(X_1(M)), M \geq 1$.
\end{enumerate}

To conclude this section, we would like to mention that the relations between visibility, the Birch and Swinnerton-Dyer special value formula, and congruences of modular forms, at the heart of the results presented below, were explored in a series of articles by Agashe, cf. for instance \cite{aga10}.

\section{Modularity of $\Sh(E/\Q)[p]$}\label{mod-sha-sect}

\subsection{Statements} The aim of this section is to prove new instances of Conjecture \ref{mod-sha-conj}, as in Corollary \ref{cor-jsconj} below. Our main result is the following.
\begin{thm}\label{mainthm}
Let $E/\Q$ be an elliptic curve of squarefree conductor $N$, without complex multiplication. Let $K$ be an imaginary quadratic field such that the prime factors $q\mid N$ are unramified in $K$, and write $N=N^-N^+$, where $q \mid N^-$ (resp. $q \mid N^+$) if $q$ is inert (resp. split) in $K$. Let $p>3$ be a prime not dividing $N\mathrm{disc}(K/\Q)$. Assume that $\bar{\rho}: \mathrm{Gal}(\bar{\Q}/\Q)\rightarrow \mathrm{Aut}_{\mathbf{F}_p}(E[p])$ has the following properties:
\begin{enumerate}
\item $\bar{\rho}$ is surjective;
\item if $q \mid N^-$ and $q \equiv \pm 1 \pmod p$ then $\bar{\rho}$ is ramified at $q$;
\item if $q \mid N^+$ then $\bar{\rho}$ is ramified at $q$; furthermore, if $E$ is supersingular at $p$, there is $q \mid N^+$ such that the restriction of $\bar{\rho}$ to $\mathrm{Gal}(\bar{\Q}_q/\Q_q)$ is a ramified extension of $\mu$ by $\mu\chi_{cycl}$, where $\mu$ is the non-trivial unramified quadratic character of $\mathrm{Gal}(\bar{\Q}_q/\Q_q)$.
\end{enumerate}
Then there is a $J_0$-modular abelian variety $A/\Q$ and an embedding $\iota: E \rightarrow A$ such that every $c\in \Sh(E/K)[p]$ is visible in $A\times_\Q K$ via $\iota \times_\Q K$.
\end{thm}
Before discussing the proof of the theorem, let us deduce a similar result over $\Q$.
\begin{corol}\label{cor-jsconj}
Let $E/\Q$ be an elliptic curve of squarefree conductor $N$, without complex multiplication. Let $p>3$ be a prime not dividing $N$. Assume that $\bar{\rho}: \mathrm{Gal}(\bar{\Q}/\Q)\rightarrow \mathrm{Aut}_{\mathbf{F}_p}(E[p])$ has the following properties:
\begin{enumerate}
\item $\bar{\rho}$ is surjective;
\item if $q \mid N$ and $q \equiv \pm 1 \pmod p$ then $\bar{\rho}$ is ramified at $q$;
\item if $E$ is supersingular at $p$, there is $q \mid N$ such that the restriction of $\bar{\rho}$ to $\mathrm{Gal}(\bar{\Q}_q/\Q_q)$ is a ramified extension of $\mu$ by $\mu\chi_{cycl}$, where $\mu$ is the non-trivial unramified quadratic character of $\mathrm{Gal}(\bar{\Q}_q/\Q_q)$.
\end{enumerate}
Then every class in $\Sh(E/\Q)[p]$ is modular.
\end{corol}
\begin{proof}
Choose an imaginary quadratic field $K$ of discriminant coprime with $Np$, and such that $q \mid N$ is split (resp. inert) in $K$ if $\bar{\rho}$ is ramified (resp. unramified) at $q$. Then $E/K$ satisfies the assumptions of the previous theorem, so there is a $J_0$-modular abelian variety $A\supset E$ such that every class $c \in \Sh(E/K)[p]$ is visible in $A\times_\Q K$. As the restriction map $\Sh(A/\Q)[p]\rightarrow \Sh(A/K)[p]$ is injective, every class in $\Sh(E/\Q)[p]\subset\Sh(E/K)[p]$ is visible in $A$.
\end{proof}

\subsection{Visibility of the classes $z_\mathfrak{l}^{BD}\in H^1(K, E[p])$} Our aim is to prove Theorem \ref{mainthm}; the assumptions and notation of the theorem are in force from now on. We will work with Bertolini--Darmon's bipartite Kolyvagin system $\{z_\mathfrak{l}^{BD}, \mathfrak{l} \in \mathcal{A}\}$ for $E[p]$, introduced in \cref{bdks}. The key ingredient is the following proposition.
\begin{prop}\label{keyprop-vis}
For every $\mathfrak{l}\in \mathcal{A}^-$, the image of $z_\mathfrak{l}^{BD}\in H^1(K, E[p])$ in $H^1(K, E)$ is visible in a quotient of $J_0(N\mathfrak{l})$.
\end{prop}

\subsubsection{Construction of the classes $z_\mathfrak{l}^{BD}\in H^1(K, E[p])$}\label{const-zl} Fix $\mathfrak{l}\in \mathcal{A}^-$. Let us first describe precisely the construction of the class $z_\mathfrak{l}^{BD}\in H^1(K, E[p])$. As in \cref{levrais} we have a newform $f_{\mathfrak{l}}\in S_2(\Gamma_0(N\mathfrak{l}))$, a prime ideal $\mathfrak{p}_\mathfrak{l}$ of the ring $\mathcal{O}_{\mathfrak{l}}\subset \C$ generated by its Fourier coefficients, and a prime ideal $\bar{\mathfrak{p}}_\mathfrak{l}\subset \bar{\mathcal{O}}_\mathfrak{l}$ such that $\bar{\mathfrak{p}}_\mathfrak{l}\cap \mathcal{O}_\mathfrak{l}=\mathfrak{p}_\mathfrak{l}$. Let $\mathfrak{m}_\mathfrak{l}$ be the kernel of the $\mathcal{O}_\mathfrak{l}/\mathfrak{p}_\mathfrak{l}$-valued morphism of the Hecke algebra attached to $f_\mathfrak{l}$. Let $J_\mathfrak{l}$ be the Jacobian of the Shimura curve $\mathrm{Sh}_{N^+}(B^\times_\mathfrak{l})$ (compactified if $N^-\mathfrak{l}=1$) of level $\Gamma_0(N^+)$ attached to the quaternion algebra ramified at $N^-\mathfrak{l}$. By \cite[Proposition 4.4]{pw11} we have an isomorphism of $\mathrm{Gal}(\bar{\Q}/\Q)$-modules $J_\mathfrak{l}[p]/\mathfrak{m}_\mathfrak{l}\simeq E[p]$ (unique up to multiplication by a unit in $\Z/p\Z$). Furthermore, $J_\mathfrak{l}(K)/\mathfrak{m}_\mathfrak{l}\simeq \mathrm{Pic}(\mathrm{Sh}_{N^+}(B^\times_\mathfrak{l}))(K)/\mathfrak{m}_\mathfrak{l}$ since $\mathfrak{m}_\mathfrak{l}$ is non-Eisenstein, hence the $p$-descent exact sequence for $J_\mathfrak{l}(K)$ yields a map
\begin{equation*}
\mathrm{Pic}(\mathrm{Sh}_{N^+}(B^\times_\mathfrak{l}))(K)/\mathfrak{m}_\mathfrak{l}\rightarrow H^1(K, J_\mathfrak{l}[p]/\mathfrak{m}_\mathfrak{l})\simeq H^1(K, E[p]).
\end{equation*}
The class $z_\mathfrak{l}^{BD}$ is defined to be the image of the divisor $\sum_{Q \in CM(\mathcal{O}_K)}Q$.

\subsubsection{Proof of Proposition \ref{keyprop-vis}}\label{mainproof} If $f_\mathfrak{l}$ has rational Fourier coefficients, so that via the Eichler--Shimura construction one can attach to it an elliptic curve $E_\mathfrak{l}$ quotient of $J_\mathfrak{l}$, then the visibility of $z_\mathfrak{l}^{BD}$ in $(E \times E_\mathfrak{l})/E[p]$ follows from the argument in \cref{vis-cong}. The general case will be established via a variation of that argument. 

We will construct an abelian variety $A_\mathfrak{l}$ with an action of $\bar{\mathcal{O}}_\mathfrak{l}$, isogenous to the quotient of $J_\mathfrak{l}$ attached to (the Jacquet--Langlands transfer of) $f_\mathfrak{l}$. We will also prove the existence of a surjective map $J_\mathfrak{l} \rightarrow A_\mathfrak{l}$ giving rise to a non-zero map $J_\mathfrak{l}[p]\rightarrow A_\mathfrak{l}[\bar{\mathfrak{p}}_\mathfrak{l}]$ factoring through $J_\mathfrak{l}[p]/\mathfrak{m}_\mathfrak{l}$, and yielding a commutative diagram:
\begin{equation}\label{com-diag}
\begin{tikzcd}
H^1(K, J_\mathfrak{l}[p]) \arrow[r] \arrow[d]
& H^1(K, J_\mathfrak{l})[p] \arrow[dd] \\
H^1(K, J_\mathfrak{l}[p]/\mathfrak{m}_\mathfrak{l}) \arrow[d]
& \\
H^1(K, A_\mathfrak{l}[\bar{\mathfrak{p}}_\mathfrak{l}]) \arrow[r] & H^1(K, A_\mathfrak{l})[\bar{\mathfrak{p}_\mathfrak{l}}].
\end{tikzcd}
\end{equation}
Admitting the existence of such a diagram for the moment, let us conclude the proof of Proposition  \ref{keyprop-vis}. The map $J_\mathfrak{l}[p]/\mathfrak{m}_\mathfrak{l}\rightarrow A_\mathfrak{l}[\bar{\mathfrak{p}}_\mathfrak{l}]$ is injective (as it is non-zero and the source is an irreducible $\mathrm{Gal}(\bar{K}/K)$-module). Calling $\Phi$ the image, we have $\Phi \simeq E[p]$, and the class $z_\mathfrak{l}^{BD}$ is visible in $(E \times A_\mathfrak{l})/\Phi$ (as usual, the action of $\Phi$ is anti-diagonal). Indeed, the image of $\sum_{Q \in CM(\mathcal{O}_K)}Q$ in $H^1(K, J_\mathfrak{l}[p])$ goes to zero via the top horizontal map in \eqref{com-diag}, hence the image of $z_\mathfrak{l}^{BD}$ in $H^1(K,  A_\mathfrak{l}[\bar{\mathfrak{p}}_\mathfrak{l}])$ goes to zero in $H^1(K, A_\mathfrak{l})$. Visibility of $z_\mathfrak{l}^{BD}$ in $(E \times A_\mathfrak{l})/\Phi$ follows from the commutative diagram
\begin{center}
\begin{tikzcd}
\Phi \arrow[r] \arrow[d]
& A_\mathfrak{l} \arrow[d] \\
E \arrow[r]
& (E\times A_\mathfrak{l})/\Phi.
\end{tikzcd}
\end{center}
Finally, the abelian variety $A_\mathfrak{l}$ admits a surjective morphism from the $\mathfrak{l}$-new quotient of $J_0(N\mathfrak{l})$; as $E$ is a quotient of the $\mathfrak{l}$-old part of $J_0(N\mathfrak{l})$, we conclude that there is a surjection $J_0(N\mathfrak{l})\rightarrow (E\times A_\mathfrak{l})/\Phi$.

\subsubsection{Construction of the diagram \eqref{com-diag}}\label{heart-pf} Let $I_\mathfrak{l}$ be the kernel of the $\mathcal{O}_\mathfrak{l}$-valued morphism of the Hecke algebra attached to $f_\mathfrak{l}$. As in \cite[\S 3.7]{zha14}, we take $A_\mathfrak{l}$ to be an abelian variety in the isogeny class of $J_\mathfrak{l}/I_\mathfrak{l}$ with a surjection $\varphi: J_\mathfrak{l}\rightarrow A_\mathfrak{l}$ factoring through $I_\mathfrak{l}$, and such that the induced $\mathcal{O}_\mathfrak{l}$-action on $A_\mathfrak{l}$ extends to an action of $\bar{\mathcal{O}}_\mathfrak{l}$. Let $e$ be the maximum exponent of $\bar{\mathfrak{p}}_\mathfrak{l}$ dividing $p\bar{\mathcal{O}}_\mathfrak{l}$, and write $A_\mathfrak{l}[p]=A_\mathfrak{l}[\bar{\mathfrak{p}}_\mathfrak{l}^e]\times T$. Let $\alpha \in \bar{\mathcal{O}}_\mathfrak{l}$ be an element with $\bar{\mathfrak{p}}_\mathfrak{l}$-adic valuation 1, and with valuation 0 at the other $p$-adic places of $\bar{\mathcal{O}}_\mathfrak{l}$. The idea is that the bottom map in \eqref{com-diag} should arise from the long exact sequence coming from ``multiplication by $\bar{\mathfrak{p}}_\mathfrak{l}$''. Even though the ideal $\bar{\mathfrak{p}}_\mathfrak{l}$ might not be principal, the desired map can be constructed starting from the map $H^1(K, A_\mathfrak{l}[p])\rightarrow H^1(K, A_\mathfrak{l})[p]$, multiplying by $\alpha^{e-1}$ and projecting onto the $\bar{\mathfrak{p}}_\mathfrak{l}^e$-torsion part.

Let us make this idea precise. As in \cite[\S 3.7]{zha14}, we take $\varphi$ such that the image of the induced map between $p$-adic Tate modules is not contained in $\bar{\mathfrak{p}}_\mathfrak{l}T_p(A_\mathfrak{l})$. Then the image of the composite $J_\mathfrak{l}[p]\rightarrow A_\mathfrak{l}[p]\rightarrow A_\mathfrak{l}[\bar{\mathfrak{p}}_\mathfrak{l}^e]$ is not contained in $\bar{\mathfrak{p}}_\mathfrak{l}A_\mathfrak{l}[\bar{\mathfrak{p}}_\mathfrak{l}^e]$, so it is not killed by $\alpha^{e-1}$. We have the following commutative diagram, where the bottom vertical maps are induced by multiplication by $\alpha^{e-1}$ and projection onto the $\bar{\mathfrak{p}}_\mathfrak{l}^e$-torsion part:
\begin{center}
\begin{tikzcd}
H^1(K, J_\mathfrak{l}[p]) \arrow[r] \arrow[d, "\varphi"]
& H^1(K, J_\mathfrak{l})[p] \arrow[d, "\varphi"] \\
H^1(K, A_\mathfrak{l}[p]) \arrow[d, "\alpha^{e-1}"] \arrow[r]
&  H^1(K, A_\mathfrak{l})[p] \arrow[d, "\alpha^{e-1}"]\\
H^1(K, A_\mathfrak{l}[\bar{\mathfrak{p}}_\mathfrak{l}]) \arrow[r] & H^1(K, A_\mathfrak{l})[\bar{\mathfrak{p}}_\mathfrak{l}].
\end{tikzcd}
\end{center}
Let $\psi$ be the composite of the maps $J_\mathfrak{l}[p]\xrightarrow{\varphi}A_\mathfrak{l}[p]\xrightarrow{\alpha^{e-1}} A_\mathfrak{l}[\mathfrak{p}_\mathfrak{l}]$. Our choice of $\varphi$ ensures that $\psi$ is non-zero. Furthermore, it factors through $I_\mathfrak{l}$ because $\varphi$ does. The target is killed by $\bar{\mathfrak{p}}_\mathfrak{l}$, hence by  $\bar{\mathfrak{p}}_\mathfrak{l}\cap \mathcal{O}_\mathfrak{l}=\mathfrak{p}_\mathfrak{l}$. Therefore, the map $\psi$ factors through $\mathfrak{m}_\mathfrak{l}$, yielding \eqref{com-diag}.

\subsubsection{Proof of Theorem \ref{mainthm}} We may, and do, assume that $r=\dim(\mathrm{Sel}(K, E[p]))\geq 2$ (if $r=0$ there is nothing to prove; for $r=1$, Theorem \ref{bd-gensel} ensures that $\mathrm{Sel}(K, E[p])$ is generated by $z_1^{BD}$, which comes from a $K$-point, hence $\Sh(E/K)[p]=0$). Take $\mathfrak{l}=\ell_1\ell_2\cdots \ell_r \in \mathcal{A}$ as in Theorem \ref{bd-gensel}. For $1 \leq i \leq r$, let $\mathfrak{l}^{(i)}=\mathfrak{l}/\ell_i \in \mathcal{A}^-$, and let $A_{\mathfrak{l}^{(i)}}$ be the abelian variety constructed in \cref{heart-pf}. Let 
\begin{equation*}
A=(E\times A_{\mathfrak{l}^{(1)}}\times \cdots \times A_{\mathfrak{l}^{(r)}})/E[p]^r,
\end{equation*}
where the embedding $E[p]^r\rightarrow E\times A_{\mathfrak{l}^{(1)}}\times \cdots \times A_{\mathfrak{l}^{(r)}}$ sends $(x_1, \ldots, x_r)$ to $(-x_1-x_2- \cdots -x_r, x_1, x_2, \ldots, x_r)$. The map $E\rightarrow A$ sending $x$ to $(x, 1, \ldots, 1)$ is an embedding, and for every $1 \leq i \leq r$ we obtain a diagram
\begin{center}
\begin{tikzcd}
E[p] \arrow[r] \arrow[d] & A_{\mathfrak{l}^{(i)}} \arrow[d] &\\
 E \arrow[r] & (E \times A_{\mathfrak{l}^{(i)}})/E[p] \arrow[r] & A.
\end{tikzcd}
\end{center}
By the proof of Proposition \ref{keyprop-vis} we know that every $z_{\mathfrak{l}^{(i)}}^{BD} \in \mathrm{Sel}(K, E[p])$ yields a class in $\Sh(E/K)[p]$ which is visible in $(E \times A_{\mathfrak{l}^{(i)}})/E[p]$, hence in $A$. Therefore, all linear combinations of the images of the classes $z_{\mathfrak{l}^{(i)}}^{BD}$ in $\Sh(E/K)[p]$ are visible in $A$. Theorem \ref{bd-gensel} implies that every class in $\Sh(E/K)[p]$ is visible in $A$.

Finally, fix $1 \leq j \leq r$. The abelian variety $A_{\mathfrak{l}^{(j)}}$ (isogenous to the quotient of $J_0(N\mathfrak{l}^{(j)})$ attached to $f_{\mathfrak{l}^{(j)}}$) is the only one among the $A_{\mathfrak{l}^{(i)}}$ which is a quotient of the $\mathfrak{l}^{(j)}$-new part of $J_0(N\mathfrak{l})$. On the other hand, $E$ is a quotient of the $\mathfrak{l}$-old part of $J_0(N\mathfrak{l})$. Therefore, there is a surjection $J_0(N\mathfrak{l})\rightarrow A$.

\begin{rem}
From the proof of Theorem \ref{mainthm} we obtain the following additional information: if $r=\dim(\mathrm{Sel}(K, E[p]))$, then there is $\mathfrak{l}=\ell_1\cdots \ell_r\in \mathcal{A}$ such that $\Sh(E/K)[p]$ is visible in a quotient of $J_0(N\mathfrak{l})$. It would of course be more satisfying to be able to replace $r$ by (a number only depending on) $\dim(\Sh(E/K)[p])$, but we are unable do this as our argument does not distinguish $E(K)/pE(K)$ from $\Sh(E/K)[p]$ (cf. \cref{crit}).
\end{rem}

\subsection{Example}\label{concrete} Let us consider the elliptic curve (with Cremona label 307010b1)
\begin{equation*}
E: Y^2+XY+Y=X^3+511495021X+14412635791156
\end{equation*}
of conductor $N=2\cdot 5\cdot 11 \cdot 2791$. According to the LMFDB database \cite[\href{https://www.lmfdb.org/EllipticCurve/Q/307010/b/2}{Elliptic curve 307010.b2}]{lmfdb}, this is an optimal quotient of $J_0(N)$, with modular degree $m_E=435265488$, analytic rank $0$ and analytic order of $\Sh(E/\Q)$ equal to $169$ (note that the $p$-part of the Birch and Swinnerton-Dyer special value formula for $E/\Q$ holds for $p=13$ by \cite[Theorem 2]{su14}). As $m_E\not \equiv 0 \pmod {13}$, by the discussion in \cref{vis-modav} the classes of order 13 in $\Sh(E/\Q)$ are not visible in $J_0(N)$. On the other hand, for $p=13$ the assumptions of Corollary \ref{cor-jsconj} are satisfied: (1) follows from \cite[Theorem 4]{maz78}, (2) holds as no prime $q \mid N$ is congruent to $\pm 1 \pmod {13}$, and $(3)$ holds because $E$ is ordinary at $13$. Therefore $\Sh(E/\Q)[13]$ is visible in a quotient of $J_0(N\mathfrak{l})$ for some $1 \neq \mathfrak{l} \in \mathcal{A}$ (which we did not attempt to find explicitly).

\bibliographystyle{amsalpha}
\bibliography{shavis}

\end{document}